\documentclass{amsart}
\usepackage{amssymb,latexsym}
\usepackage{graphicx}
\newtheorem{theorem}{Theorem}

\newtheorem{definition}{Definition}

\begin{document}
\title{A RAIKOV-TYPE THEOREM FOR RADIAL POISSON DISTRIBUTIONS:\\
 A PROOF OF KINGMAN'S CONJECTURE.}
\author{Thu Van Nguyen}
\address{Department of Mathematics;
         International University, HCM City;
         No.6 Linh Trung ward, Thu Duc District, HCM City;
         Email: nvthu@hcmiu.edu.vn}
\date{March 24, 2011}

\begin{abstract}  In the present paper we prove the following conjecture in Kingman, J.F.C., Random walks with spherical symmetry, Acta Math.,109, (1963), P. 11-53,
 concerning a famous Raikov's theorem of decomposition of Poisson random variables: 
 "If a radial sum of two independent random variables X and Y is radial Poisson, then each of them must be radial Poisson."
\end{abstract}
 
\maketitle{Keywords and phrases: radial characteristic functions; radial Poisson distribution.

AMS2000 subject classification: 60B99, 60E07, 60E99.}
 \vskip 1cm
 \section{Introduction, Notations and Preliminaries}\label{S:intro}
It is well-known that  for each $s\ge -\frac{1}{2}$ the Bessel function $\Lambda_s(.)$ defined by
\begin{equation}\label{eq:Lam}
 \Lambda_s(x)= \Gamma(s+1) J_{s}(x)/(\frac{1}{2}x)^{s},
\end{equation}
where $J_s(x)$ denotes the Bessel function of the first kind,
\begin{equation}\label{eq:Bessel}
J_s(x):= \Sigma_{j=0}^{\infty} \frac{(-1)^j(x/2)^{\nu+2j}}{j!\Gamma(\nu+j+1)}
\end{equation}
can be used as a kernel of {\it radial charachteristic functions} (rad.ch.f.'s) which possesses many of the properties usually associated with the familiar univariate charachteristic functions (ch.f.'s). In particular, in terms of Bessel functions, one can introduce important classes of probability distributions in the context of the Kingman convolutions such as {\it radial Gaussian, radial Poisson distributions, etc.   

  Let $\mathcal P:=\mathcal P(\mathbb R^+)$  denote the set of all probability measures (p.m.'s) on the positive half-line   $\mathbb R^+$ equipped with the weak convergence. The Kingman convolution (cf. Kingman \cite{Ki}, Urbanik \cite{U1}) is defined as follows. For each continuous bounded function f on $\mathbb R^{+}$
  we put
 \begin{multline}\label{astKi}
\int_{0}^{\infty}f(x)\mu\ast_{1,\delta}\nu(dx)=\frac{\Gamma(s+1)}{\sqrt{\pi}\Gamma(s+\frac{1}{2})}\\
\int_{0}^{\infty}\int_{0}^{\infty}\int_{-1}^{1}f((x^2+2uxy+y^2)^{1/2})(1-u^2)^{s-1/2}\mu(dx)\nu(dy)du,
\end{multline}
  where  $ \mu\mbox{ and }\nu\in\mathcal P\mbox{ and }\delta=2(s+1)\geq1$ . In the sequel, for the sake of simplicity, we will denote  $\ast_{1,\delta}=\ast_s.$ The convolution algebra $(\mathcal{P},\ast_s)$ is
 the most important example of Urbanik convolution algebras. In language of the
  Urbanik convolution algebras, the {\it characteristic measure}, say $\sigma_s$, of the Kingman convolution
  has the Rayleigh density
  \begin{equation}\label{Ray}
  d\sigma_s(y)= \frac{2{(s+1)^{s+1}}}{\Gamma(s+1)}y^{2s+1}\exp{(-(s+1)y^2)}dy
  \end{equation}
with the characteristic exponent $\varkappa=2.$    It is known (cf. Kingman \cite{Ki}, Theorem 1), that the kernel $\Lambda_s$ itself is an ordinary ch.f. of a symmetric p.m., say $F_s$, defined on the
interval [-1,1]. Thus, if $\theta_s$ denotes a random variable (r.v.) with distribution
$F_s$ then for each $t\in \mathbb R^+$,
\begin{equation}\label{eq:LamThe}
\Lambda_s(t)= E\exp{(it\theta_s)}=\int_{-1}^1\cos{(tx)}dF_s(x).\end{equation}
 Suppose that $X$ is a nonnegative r.v. with distribution $\mu\in\mathcal{P}$
 and $X$ is independent of $\theta_s$.  The {\it radial characteristic function}
 (rad.ch.f.) of $\mu$, denoted by $\hat\mu(t),$ is defined by
  \begin{equation}\label{ra.ch.f.}
\hat\mu(t) = E\exp{(itX\theta_s)} = \int_0^{\infty}
\Lambda_s(tx)\mu(dx),
\end{equation}
 for every $t\in \mathbb R^{+}$.
 The characteristic measure of the Kingman convolution $\ast_s$, denoted by $\sigma_s$,
has the Maxwell density function
\begin{equation}\label{Maxwell density}
\frac{d\sigma_s(x)}{dx}=\frac{2(s+1)^{s+1}}{\Gamma(s+1)}x^{2s+1}exp\{-(s+1)x^2\}, \quad(0<x<\infty).
\end{equation}
and the rad.ch.f.
\begin{equation}
\hat\sigma_s(t)=exp\{-t^2/4(s+1)\}.
\end{equation}
\section{Transforms $\tau_s, s\ge -\frac{1}{2}$}
 Let us denote by ${\mathbf S}$ the class of all symmetric p.m.'s on ${\mathbb R}$. Obviously, this class of p.m.'s is closed with respect to the ordinary convolution $\star$ and convex combinations of p.m.'s and the weak convergence.
 Define transforms  $\tau_s: \mathcal P\to {\mathbf S},\;  s\ge -\frac{1}{2}$ as follows
\begin{equation}\label{tau_s transform}
 \tau_s(\mu)(E)=\int_0^{\infty}T_c(F_{s})(E) \mu(dc),
 \end{equation}
 where $c\ge0\mbox{ and } E\subset \mathbb {R}\mbox{ and for a p.m. }\mu\in \mathbb{R}$
 $$T_c\mu(E)=\mu(c^{-1}E).$$
 By Sonine's first finite integral for Bessel function (\cite{Wat}, p. 373), for $s>v\ge-\frac{1}{2}$ we have
 \begin{equation} 
 \Lambda_s(t)=\frac{2\Gamma(s+1)}{\Gamma(v+1)\Gamma(s-v))}\int_0^1x^{2v+1}(1-x^2)^{s-v-1}\Lambda_v(tx)dx,
 \end{equation}
 which implies that the transform $\tau_s$ is subordinate to $\tau_v$. In particular, each transform $\tau_s, \tau\ge-\frac{1}{2}$ is subordinate to $\tau_{-\frac{1}{2}}$ which is defined via the cosine convolution. Specificly, for $s>v\ge-\frac{1}{2}\mbox{ and } \mu\in \mathcal P,$ 
 \begin{eqnarray}
 \hat{\mu}(t)&=&\int_0^{\infty}\Lambda_s(tx)\mu(dx)=\int_0^{\infty}cos(tx)(\tau_s\mu)(dx)\notag\\ 
 {\empty}&=&\int_0^{\infty}\frac{2\Gamma(s+1)}{\Gamma(v+1)\Gamma(s-v)}\int_0^1u^{2v+1}(1-u^2)^{s-v-1}\Lambda_v(tux)du\; \mu(dx)\\
 {\empty}&=&\int_0^{\infty}\Lambda_v(tx)[\frac{2\Gamma(s+1)}{\Gamma(v+1)\Gamma(s-v)}\int_0^1u^{2v+1}(1-u^2)^{s-v-1} \mu(dx/u)du]\notag \\  
 {\empty}&=&\int_0^{\infty}cos(tx)(\tau_v\nu)(dx), \notag                                                        
 \end{eqnarray}
 where $\nu(dx):=\frac{2\Gamma(s+1)}{\Gamma(v+1)\Gamma(s-v)}\int_0^1u^{2v+1}(1-u^2)^{s-v-1} \mu(dx/u)du$, which implies that for every bounded continuous function on $\mathbb R^+$
 \begin{multline}\label{relationtrans}
 \int_0^{\infty}f(x)(\tau_s\mu)(dx)=\\
 \frac{2\Gamma(s+1)}{\Gamma(v+1)\Gamma(s-v)}
 \int_0^{\infty}f(x)\int_0^1(\tau_v\mu)(dx/u)u^{2v+1}(1-u^2)^{s-v-1}du\}.
 \end{multline}
 
    Let $\mathbf S_s$ be a subclass of $\mathbf S$ consisted of p.m.'s
of the form (\ref {tau_s transform}). By (\ref {tau_s transform}) and by Fourier transforms and rad.ch.f.'s and, by a similar proof of Proposition 1b in Bingham \cite{B1}, p. 176.
\begin{theorem}\label{property of homeomorphism}
For any $\mu, \nu\in \mathcal P \mbox{ and } \alpha\ge 0, \beta\ge 0, \alpha + \beta =1,$
\begin{eqnarray}\label{propert y of  F_s}
\mathcal \tau_s(\alpha\mu + \beta\nu )&=&\alpha \tau_s(\mu)+\beta \tau_s(\nu)\\
\mathcal \tau_s(\mu\ast_{s}\nu)&=&( \tau_s\mu)\star(\tau_s\nu)\\
 \tau_s(\sigma_s)&=&N(0, 2(s+1)).
\end{eqnarray}
where the $\star$ denotes the ordinary convolution.

 Hence, it follows that each pair $( \mathbf S_s, \ast,  s\ge -1/2),$ is a topological semigroup of symmetric p.m.'s with the unit element $\delta_0.$
Moreover, the map $\mathcal \tau_s$ stands for a homeomorphism between the two convolution algebras.
$(\mathcal P, \ast_s)\mbox{ and }(\mathbf S_s, \star).$ 
Finally, by virtue of (\ref{relationtrans}) the family $ {\mathbf S}_s, s\ge -\frac{1}{2}$  is monotone decreasing in s. The sets  ${\mathbf S}_s$ are continuous in s in that
\begin{equation}\label{Cont Moni}
\cap_{0\le u<s} {\mathbf S}_u=  {\mathbf S}_s, \quad \cup_{s<v\le\infty}{\mathbf S}_v={\mathbf S}_s, \quad (s\in[0, \infty]).
\end{equation}
   \end{theorem} 
  \section{Raikov's type theorem for radial Poisson distributions}
  
 Given nonnegative r.v.'s  $X\mbox{ and }Y$ with the corresponding distributions $\mu\mbox{ and }\nu$ such that $X, Y\mbox{ and }\theta_s$ are independent. Following Kingman (\cite{Ki}, p.19) we put
 \begin{equation}\label{rad.sum}
 X\oplus_sY:\overset{d}{=}\sqrt{X^2+Y^2+2XY\theta_s}
\end{equation} 
and call it (and any one of the equivelent r.v.'s $X\oplus_sY$) a {\it radial sum} of X and Y. It should be noted that
$$\mu\ast_s\nu\overset{d}{=}(X\oplus_sY).$$
In a recent paper \cite{Ng5} a multi-dimensional analogue of the Cram\'er-L\'evy theorem was obtained (see also Urbanik \cite{U2} for one-dimensional case). For the sake of simplicity, we state below the univariate version. 
\begin{theorem} 
Let $X\mbox{ and }Y$ be nonnegative independent r.v.'s such that  
  \begin{equation}\label {radial sum 2}
  \sigma_s\overset{d}{=}(X\oplus_sY).
  \end{equation}
Then, $X=\alpha X_1\mbox{ and } Y=\beta Y_1$ for some nonnegative constants  $\alpha\mbox{ and } \beta\mbox{ with } \alpha^2+\beta^2=1\mbox{ and }\sigma_s\overset{d}{=}X_1\overset {d}{=}X_2.$
\end{theorem}
It is remarkable that Rayleigh distributions share important properties with ordinary normal distributions. Our further aim is to study a similar case for radial Poisson distributions which are defined as follows.
 \begin{definition}\label{radial Poisson}(cf.Kingman \cite{{Ki}}, P. 33-34)
 A r.v. X or its distribution  $\pi$ is said to be {\it radial Poisson}, or precisely, {\it s-radial Poisson},  if its rad.ch.f. is of the form
 \begin{eqnarray}\label{rad.Poisson ch.f.}
\hat{\pi}(t)&=&E(\Lambda_s(tX))=\exp a\{\Lambda_s(ct)-1\}\quad(a, c>0\mbox{ and } t\ge 0)\\
\label{repres Poisson}
X&\overset{d}{=}&(X_1\oplus_sX_2\oplus_s\ldots\oplus_s X_N)
 \end{eqnarray} where all the $X_i$ are equal to c with probability one, and $N$ has a Poisson distribution with mean $a, \mbox{ and r.v.'s  }N, X_i, i=1, 2, \ldots $
 are independent.
 \end{definition}
 By virtue of (\ref{rad.Poisson ch.f.}, \ref{repres Poisson}) it follows that each radal Poisson distribution defined above is uniquely determined
  by a and $ c, a=EN\mbox{ and }c$ is a scale parameter and will be denoted by $\pi(s, a, c)$.
If $s=-\frac{1}{2}\mbox{ and }c=1\mbox{ we have } \ast_s=\star\mbox{ and }\pi_{-\frac{1}{2}}$ becomes an ordinary Poisson distribution.
 The distribution has atoms at 0 and at c, together with an absolutely continuous component in $x>0.$
 \begin{definition}\label{symm Poisson}
  A symmetric distribution $\nu\in \mathbf {S}$ is called {\it symmetric Poisson}, if its ch.f. (Fourier transform) is of the form
\begin{equation}\label{symmetric Poisson}
\mathcal F_{\nu}(t)= \exp a\{\cos(ct)-1\}\quad(a, c>0\mbox{ and } t\ge 0).
\end{equation}
 \end{definition}
 
  A question posed by Kingman (\cite{Ki}, P. 34) has been existing for many years: {\it" It would be interesting to prove for this distribution analogues of some of the well-known results (such as Raikov's theorem, \cite{Lu}, P.174) about the Poisson distribution".} 
  Our main aim in this paper is to give a complete answer to Kingman's question.
  \begin{theorem}\label{Raikov-type theorem}
  Suppose, $X, Y, \mbox{ and }Z$ are nonnegative independent r.v.'s such that $\pi(s, a, c)\overset{d}{=}X$ and the following equation holds
  \begin{equation}\label{radial sum}
  X\overset{d}{=}(Y\oplus_sZ).
  \end{equation}
  Then, there exist nonnegative numbers $a_1\mbox{ and }a_2\mbox{ such that }a=a_1+a_2$ and r.v.'s $X\mbox{ and }Y$ have radial Poisson distributions $\pi(s, a_1, c)\mbox{ and }\pi(s, a_2, c)$, respectively.
  \end{theorem}
  \begin{proof}  First, observe that if $X\overset{d}{=}\pi(s, a, c), \mbox{ then } X/c\overset{d}{=}\pi(s, a, 1)$. Therefore, without loss of generality, one may assume that $c=1$.   
Next, consider the case $s=-\frac{1}{2}.$ It is evident that the Urbanik's symmetric convolution  $\ast_{1, 1}$ (cf. Urbanik\cite{U1}) is identical with the Kingman convolution $\ast_{-\frac{1}{2}}$ or cosine convolution. Since,  the transform $\tau_{-\frac{1}{2}}$ is the symmetrization of p.m.'s  in $\in \mathcal P,$ it follows that every symmetric distribution  $H\in \mathbb R$ is uniquely represented by\ 
\begin{equation}\label{Sym.p.m.}
 H=\tau_{-\frac{1}{2}}(G),\quad G\in\mathcal P
\end{equation}
which implies that $\mathbf {S}=\{\tau_{-\frac{1}{2}}(G), G\in \mathcal P\}.$

Moreover, the rad.ch.f. of a distribution $G\in(\mathcal{P},\ast_{-\frac{1}{2}})$ is the same as the Fourier transform of its symmetrisation, that is
\begin{equation}\label{rad-Fourier}
\hat{G}(t)=\mathcal F_{\tau_{-\frac{1}{2}}(G)}(t)\quad t\ge 0. 
\end{equation}
In particular, if $G=\pi(-\frac{1}{2}, a, c)$ the Equation (\ref {rad-Fourier}) reads
\begin{equation*}
    \hat{\pi}(-\frac{1}{2}, a, c)(t)=\int_{-1}^1\cos(ut) \pi(-\frac{1}{2}, a, c) (du)= \exp a\{\cos(ct)-1\}\quad(a, c>0\mbox{ and } t\ge 0)
    \end{equation*}
    which shows that $\tau_{-\frac{1}{2}}(\pi(-\frac{1}{2}, a, c) )$ is a symmetric Poisson distribution.

 Suppose now that for symmetric distributions $\tau_{-\frac{1}{2}}(G_i), i=1, 2$ we have
 \begin{equation*}
\tau_{-\frac{1}{2}}(\pi(-\frac{1}{2}, a, c))=\tau_{-\frac{1}{2}}(G_1)\star\tau_{-\frac{1}{2}}(G_2)
 \end{equation*}
 which implies that
 \begin{equation*}
 \pi(-\frac{1}{2}, a, c)=G_1\ast_{-\frac{1}{2}} G_2 .
 \end{equation*}

Proceedings successively, by virtue of the classical Raikov's decomposition theorem (\cite{Lu}, P.174)  we infer that the distributions $G_1\mbox{ and }G_2$ are both Poisson distributions and, consequently, the symmetric distributions $\tau_{-\frac{1}{2}}(G_1)\mbox{ and }\tau_{-\frac{1}{2}}(G_2)$ are both symmetric Poisson distributions which shows that the Raikov's decomposition theorem holds true also for symmetric distributions on $\mathbb R.$ Now, by (\ref{symmetric Poisson}, \ref{rad-Fourier}) it follows that, Theorem \ref{Raikov-type theorem} is true for the Kingman convolution algebra $(\mathcal P, \ast_{-\frac{1}{2}}).$  In general, let $\pi(s, a, c)$ be a radial Poisson distribution with the rad.ch.f. given by the equation (\ref{rad.Poisson ch.f.}) and the decomposition (\ref{radial sum}) holds. Let $G_1\mbox{ and }G_2 $ be distributions of the r.v.'s $Y\mbox{ and }Z$, respectively.
Then we have 
$$\pi(s, a, c)=G_1\ast_sG_2.$$
Applying the transform (\ref{tau_s transform}) to both sides of the above equation one get
$$\tau_s(\pi(s, a, c))=\tau_s(G_1)\star \tau_s(G_2)$$
which by virtue of the Raikov's theorem for symmetric Poisson distributions implies that the distributions $\tau_s(G_1)\mbox{ and } \tau_s(G_2)$ are symmetric Poisson distributions and consequently, since for each $s\ge -\frac{1}{2}\mbox{ and for each }G_i\in \mathcal P$
\begin{equation*}
\hat{G_i}(t)=\mathcal F_{\tau_s(G_i)}(t)= \exp a_i\{\Lambda_s(ct)-1\}\quad(a_i, c>0, i=1, 2\mbox{ and } t\ge 0)
\end{equation*}
we have $G_i=\pi(s, a_i, c), a=a_1+a_2$ and the equation (\ref{radial sum}) holds. 
\end{proof}

\end{document}